\long\def\skipit#1{} 
\def\={\,=\,}
\def\+{\,+\,}

\documentclass[french,10pt]{amsart}
\usepackage[utf8]{inputenc}
\usepackage[T1]{fontenc}
\usepackage{pstricks-add}
\usepackage{amsfonts}
\usepackage{amsmath}
\usepackage{amssymb}
\usepackage{ifthen,calc}
\usepackage{color}
\usepackage{epsfig}
\usepackage{graphicx}
\usepackage{epstopdf}
\usepackage{epsfig}
\usepackage{latexsym}
\usepackage{multicol}
\usepackage[sc]{mathpazo}
\usepackage[format=hang, margin=10pt]{caption}
\usepackage[mathscr]{eucal}
\usepackage{mathtools}
 \usepackage[utf8]{inputenc}

\newcounter{hours}
\newcounter{minutes}
\newcommand{\printtime}{
	\setcounter{hours}{\time/60}%
	\setcounter{minutes}{\time-\value{hours}*60}
	\ifthenelse{\value{hours}<10}{0}{}\thehours:%
	\ifthenelse{\value{minutes}<10}{0}{}\theminutes}

\numberwithin{equation}{section}
\numberwithin{figure}{section}
\numberwithin{table}{section}

\newtheorem{thm}{Theorem}[section]

\newtheorem{lemma}[thm]{Lemma}

\newtheorem{J-com}{JG-comment}[section]
\theoremstyle{definition}
\newtheorem{example}{Example}[section]

\newtheorem{rem}[thm]{Remark}

\def\dsum{\displaystyle\sum}

\keywords{Embedding, Exterior permutation, Interior permutation, Signed graph, Excepted genus}

\begin{document}

\title {The average Euler-genus of the vertex-amalgamation
of signed graphs}

{
\author{Yichao Chen}
\address{School of Mathematics, SuZhou University of Science and Technology, Suzhou, 215009, China}
\email{chengraph@163.com}
}  


\vskip.51cm
\noindent {Version: \printtime\quad\today\quad}
\vskip.51cm

\begin{abstract}
In this paper, 
 we first generalize a theorem for counting the number of faces of an oriented embedding of a graph that passing through a given cut-edge set [S. Stahl, Trans. Amer. Math. Soc. 259 (1980), 129--145] to all surfaces. 
Then we extend Stahl's bounds for the average genus of the vertex-amalgamation of graphs [S. Stahl, Discrete Math. 142 (1995), 235--245] to signed graphs.

\end{abstract}

\maketitle  


\section{Introduction}

A graph $G$ is often permitted to have both loops and multiple edges in graph embeddings. Since the genus is a homeomorphic invariant, we  suppose the graphs in the paper are simple. It is well known that all closed surfaces are classified into the \textit{{orientable surfaces}} $O_i$, with $i$ handles $(i\ge0)$, and the \textit{{non-orientable surfaces}} $N_j$, with $j$ crosscaps $(j>0)$. We will use $S$ to denote a surface without regard to orientability. An \textit{embedding} of  $G$ into a closed surface $S$ is a \textit{{cellular} embedding}. If a graph $G$ has an embedding in a surface $S$ with
$n$ vertices, $m$ edges and $r$ faces, the following classical \textit{ Euler's formula }tell us that $$n-m+r=2-\gamma^E
,$$

\noindent where $\gamma^E$ is the \textit{Euler-genus} of the surface $S$ and $$\gamma^E=\left\{
         \begin{array}{ll}
           2k, & \hbox{If $S=O_k$;} \\
           h, & \hbox{If $S=N_h$.}
         \end{array}
       \right.$$

 A \textit{rotation at a vertex} $v$ of a graph $G$ is a cyclic ordering of the edge-ends incident at $v$.  A\textit{ rotation system} $R$ of a graph $G$ is an assignment of a rotation at every vertex of $G$. There is an one-to-one correspondence between the orientable embeddings and rotation systems.

An \textit{embedding} of a graph $G$ on an arbitrary surface $S$ can be described combinatorially by a \textit{signed rotation system} (or
 a general rotation system) $(R,\sigma),$ where $R$ is a rotation system of $G$ and $\sigma$ is the \textit{twist-indicator}, i.e., if $\sigma(e)=-1$, then the edge $e$ is twisted; otherwise $\sigma(e)=1$ and $e$ is untwisted.  A \textit{switch} of a signed rotation system at a vertex $v$ means reversing the rotation about $v$ and changing the sign on each of its incident edges. Two signed rotation system are \textit{equivalent} if one can be transformed into another by a sequence of switches. For terms not defined in this paper, we refer to \cite{GT87}. 
 
 A \textit{permutation-partition pair} $(P,\Pi)$ consists of an arbitrary permutation $P$ and an arbitrary partition $\Pi,$ both defined over some common underlying set $S.$ It is a combinatorial generalization of  graphs and graph embeddings. The notion of a permutation-partition pair was  introduced by Stahl in \cite{Sta80}.  It was also independently rediscovered in \cite{Arc86a,Arc86b}, called IDS by Archdeacon. For undefined
 terminologies on $(P,\Pi)$, see \cite{Sta91c}.

A \textit{signed graph} $\Sigma=(G,\sigma)$ is a graph $G$ together with a mapping $\sigma$ which assigns $+1$ or $-1$ to each edge of the graph $G.$ If $\sigma(e)=1,$ we call the edge \textit{positive}, otherwise $\sigma(e)=-1$ and the edge is \textit{negative}. 
An { embedding} of a signed graph $\Sigma=(G,\sigma)$ can be described combinatorially by a \textit{signed rotation system} $(R,\sigma),$ where $R$ is the rotation system of $G,$ and $\sigma$ can be thought of as the twist-indicator. There are many researches on signed graph embeddings, see \cite{Che22,L15,LY18,Sir91a,Sir91b,SS91,Zas92,Zas93,Zas96a,Zas97b}, etc.

Two embeddings $(R_1,\sigma),$ and $(R_2,\sigma)$ of $\Sigma=(G,\sigma)$ are {equivalent} if their rotation system $R_1$ and $R_2$ are the same. Let the \textit{sample space} $\Omega$ consist of all inequivalent  embeddings $\Sigma$. The \textit{Euler-genus random variable} $\gamma^E$ is define as $$\gamma^E: \Omega\longrightarrow \{0,1,\ldots,\}.$$ If $(R,\sigma)$ embeds $\Sigma$ on $S_i$, then $\gamma^{E}(R,\sigma)=i.$ The \textit{average Euler-genus} $\gamma_{avg}^E (\Sigma)$ of a signed graph $\Sigma$ is the
expected value of the Euler-genus random variable, over all labeled
embeddings of $\Sigma,$ using the uniform distribution $$P(\Sigma)=\frac{1}{\prod_{v\in V(G)}(d_v-1)!},$$

i.e., \begin{eqnarray}\label{defini}\gamma_{avg}^E (\Sigma)=\frac{\dsum_{R \in \Omega}\gamma^{E}(R,\sigma)}{\prod_{v\in V(G)}(d_v-1)!}.\end{eqnarray}

Note that the average genus of a single graph  was introduced in \cite{GF87} by
Gross and Furst  and was further developed in \cite{GKR93,CG92a,CG92b,CG93,CGR95}. Recently, Loth, Halasz, Masa\v{r}\'{\i}k, Mohar, and \v{S}\'{a}mal \cite{LHMMS22} and  Loth and Mohar \cite{LM22} obtained tight lower bounds for the average genus. The author proved that nonorientable average genus of a random graph is close to its maximum nonorientable genus \cite{Che22b}.
  We may refer to \cite{Che10,Sta91b,Sta95b} for more topics on average genus.


Given an oriented embedding of a graph $G$, Stahl \cite{Sta76} proved a local counting theorem that provided an expression for the number of distinct faces which contain a specific vertex on their boundary in terms of the orbits of the product of two permutations. The two permutations were derived from Edmonds' rotation systems. He then used this expression to give new proofs of several classical theorems in topological graph theory, such as the Duke's interpolation theorem and  the additivity  theorem for the genus of Battle, Harary, Kodama, and Youngs. The local counting theorem was also independently rediscovered, in a more geometrical form, in \cite{GMTW16,LHMMS22}, etc. For other works related to the local counting theorem, see \cite{Sta80,Sta82,Sta91b,CGM18}, etc.

Let $U$ be a subset of the vertex set $V(G)$ of a graph $G$.  Then the {\textit{boundary} of $U$}, denoted by $\partial(U)$,  is the set of oriented edges of $G$ that emanate from a vertex of $U$ and terminate at a vertex of $V(G)-U$. 
In \cite{Sta80}, Stahl generalized the local counting theorem to the case of containing edges of $\partial U$ for an oriented embedding. Given an embedding of $G$ on an arbitrary surface $S$, here we obtain an expression for the number of distinct faces which contain the edges of $\partial U.$  Our result can be thought of as a generalization of Stahl's theorem to nonorientable embeddings. 

The paper is organized as follows. In Section 2, we  prove the counting theorem for graph embeddings. Upper and lower bounds for the average Euler-genus of the vertex-amalgamation of signed graphs is given in Section 3.


\section{The counting theorem}

\subsection{Combinatorial embeddings of graphs}  

Let $G=(V,E)$ be a graph with  vertex set $V$ and edge set $E.$ Suppose that $K=\{1,\alpha,\beta,\gamma\}$ is the \textit{Klein four-group}. Given an edge $a\in E,$ we introduce its \textit{two sides} and \textit{two ends}.  Assume $a$ itself at one end and on one side.  Let $\alpha$ be the permutation that interchanges symbols at the same end but different sides of an edge, for each edge. Let $\beta$ be the permutation that interchanges symbols at the same side but different ends of an edge, for each edge.  We call the set $Ka=\{a,\alpha a, \beta a,\gamma a \}$  a \textit{quadricell}.

An embedding $(R,\sigma)$ of $G$ on some surfaces induces a \textit{bi-rotation system} $P$ on $V$ such that at each vertex, all its incident semi-edges are with a cyclic order, called a \textit{bi-rotation}. Since each edge is considered as a quadricell, the bi-rotation at a vertex $v$ is $$\left\{(a,Pa,\ldots,P^{(m-1)}a),(\alpha a,\alpha P^{-1} a,\ldots,\alpha P^{-(m-1)}a)\right\}.$$


We denote by $\langle P, \gamma \rangle$ the group of permutations $P, \gamma$ generated by them.  Tutte provided an axiomatization for any embedding of $G$, as stated in the following result of
Tutte.

\begin{def}\label{tutte}
An \textit{embedding $M$ (or a map)} is an ordered triple $(\alpha, \beta,P)$ of permutations acting on a set $S$ of $4m$ elements, such that

\begin{enumerate}
  \item  for any $a\in S,$ then $a, \alpha a, \beta a, \gamma a$ are distinct.
  \item  $\alpha P=P\alpha^{-1}$.
  \item  for each $a\in S$, the orbits of $P$ through $a$ and  $\alpha a$ are distinct.
  \item $\langle P, \gamma \rangle$ acts transitively on $S.$
\end{enumerate}
\end{def}

It's known that if the action of $\langle P,\gamma\rangle$ on $S$ is transitive, then $M$ is non-orientable, otherwise there are precisely two orbits, and $M$ is orientable. For more on Tutte's axiomatization of graph embeddings, see \cite{Liu09,Tut84}, etc.

\subsection{The counting theorem}
Suppose that the rotation at every vertex $v$ of $G$ is counter-clockwise.  Given an oriented edge $e$, we denote $e^l$ and $e^r$ as the \textit{left side} and the \textit{right side} of $e$, respectively. We also denote $e^+$ and $e^-$ as the \textit{lower end} and the \textit{upper end} of $e$, respectively. Under this convention, $e^{l_+}$, $e^{r_+}$, $e^{l_-},$ and $e^{r_-}$ denote the \textit{lower left side}, the \textit{lower right side}, the \textit{upper left side}, and the \textit{upper right side} of $e$ respectively, as shown in Figure \ref{fig:edge}. 

 \begin{figure}[h]
\centering
\includegraphics[width=1.5cm,height=2cm]{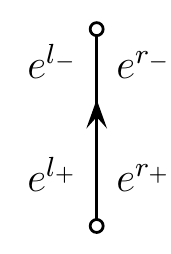}
\caption{}\label{fig:edge}
\end{figure}

Let $(R,\sigma)$ be a signed rotation system of $G$ with $\sigma(e)=1$ for $e\in \partial(U)$ and let $\iota_{(R,\sigma)}$ be the embedding of $G$ induced by $(R,\sigma)$.  Let $U\subseteq V(G)$. Any rotation system $(R,\sigma)$ of $G$  induces two permutations, \textit{exterior permutation} $Ext_{R, U}$ and \textit{interior permutation} $Int_{R, U}$ on $\partial(U)$ as follows. When we say $e^l,$ we mean its left upper end or left lower end of $e^l,$ namely, we have $e^{l}\in \{e^{l_+},e^{l_-}\}.$ Similarly, the edge $e^r$ means its right upper end or right lower end and $e^r=e^{r_+}$ or $e^{r_-}.$ For any $e\in \partial(U),$ and let $\delta\in \{+,-\}$, we define $$Ext_{R, U}(e)=\gamma[P\gamma]^{x(e)} (e^{r_{\delta}}),\ \ \ \ \  Int_{R, U}(e)=\gamma[\gamma{P}]^{y(e)} (e^{l_{\delta}})$$
where $x(e)$ and $y(e)$ are the least positive integers such that $[P\gamma]^{x(e)}(e^{r_{\delta}}) \in \{e_{s}^{l_{-}},e_{s}^{l_+}\}$ for some $e_s\in \partial(U)$, and $[\gamma P]^{x(e)} (e^{l_{\delta}})\in \{e_{t}^{r_{+}},e_{t}^{r_{-}}\}$ for some $e_t\in \partial(U)$ respectively. When $U$ consists of a vertex $u$, i.e., $U=\{u\}$, then we also denote $Ext_{R, U}$ and $Int_{R, U}$ as $Ext_{R, u}$ and $Int_{R, u}$, respectively.

We have an intuitive explanation for $Ext_{R, U}$. For each oriented edge $e$ in $\partial(U)$, let ${F}_e$ be the  face of the embedding $\iota_{(R,\sigma)}$ whose boundary contains $e$. Staring from $e^{r_{\delta}},$  then follow along the boundary of $F_e$ until the first edge $e_s^{l,-}$ or $e_s^{l,+}$  belonging to $\partial(U)$ in encountered, this edge is $Ext_{R, U}(e)$, namely, $Ext_{R, U}(e)=e_s.$


 Again, here's  an intuitive explanation for $Int_{R, U}.$ 
Staring from $e^{l_{\delta}}$, then follow along the boundary of $F_e$ until the first edge $e_{t}^{r_{+}},$ or $e_{t}^{r_{-}}$  belonging to $\partial(U)$ in encountered, this edge is $Int_{R, U}(e)$. Then $Int_{R, U}(e)=e_t$.

Here, we give a proof that $Ext_{R, U}$ is indeed a permutation of $\partial(U).$ We only need to give a proof for the case $U=\{u\}$, because we can contract $U$ into a vertex $u$ by shrinking the edges so that the resulting graph is still embedded in the same surface.

  By
definition, $Ext_{R, U}(e)= f^r$ for some $f\in \partial(U).$ 
Thus $ Ext_{R, U}(e)\in \partial(U).$
We now go to show that  $Ext_{R, U}$  is injective. If there exit two different edges $e,f\in \partial(U) $ such that $Ext_{R, U}(e)=Ext_{R, U}(f)$. By symmetry, we can assume that $x(e)$ is greater than or equal to $x(f).$ Thus $$\gamma[P\gamma]^{x(e)-x(f)-1} (e^{r_{\delta}})=P^{-1} (f^{r_{\delta}}).$$ Note that $x(e)-x(f)-1$ is less than $x(e).$

If $x(e)=x(f)$, then $$P^{-1} (f^{r_{\delta}})=P^{-1} (e^{r_{\delta}}),$$ thus $f=e$ which contradicts $e\neq f$.

If  $x(e)=x(f)+1$, then $$P\gamma(e^{r_{\delta}})=(f^{r_{\delta}}).$$ This means $e,f$ lie on the boundary of the same face and  $e, f$ are adjacent in the face, this contradicts that $G$ is a simple graph, again it's impossible.

If $x(e)-x(f)\geq 2$, then $\gamma[P\gamma]^{x(e)-x(f)-1} (e^{r_{\delta}})=P^{-1}(f^{r_{\delta}})$. If $f^{r_{\delta}}=f^{r_+},$ then the initial point of $P^{-1} (f^{r_{+}})$  is the same as that of $f$ and is therefore in $\partial(U),$ which contradicts the minimality of $x(e).$ Otherwise $\gamma[P\gamma]^{x(e)-x(f)+1} (e^{r_{-}})=P\gamma(f^{r_{-}})=Pf^{l_{+}}$, again the initial point of $P(f^{l_{+}})$  is the same as that of $f$ and is therefore in $\partial(U).$ This also implies that $x(f)\geq 2$. Thus $x(e)-x(f)+1<x(e)$, it's impossible.


In summary, $Ext_{R, U}$ is a permutation of $\partial(U).$ Similarly, we can also prove that $Int_{R, U}$ is a permutation of $\partial(U).$



Given a permutation $P$, we denote by $||P||$ the number of orbits in the permutation $P$. We have the following  counting theorem for a graph embedding.

\begin{thm} \label{Yc18}
Suppose that $ U $ is a set of vertices of a graph $G,$  and let $(R,\sigma)$ be a signed rotation system of $G$ with $\sigma(e)=1$ for any $e\in \partial(U).$ Then the number of faces of the embedding $\iota_{(R,\sigma)}$ whose boundaries contain arcs of $\partial(U)$ is $||Ext_{R,U}\circ Int_{R, U}||.$
\end{thm}

\begin{proof}Note that $||Int_{R, U}\circ Ext_{R, U}||=||Ext_{R, U}\circ Int_{R, U}||,$ we prove that the number of faces of the embedding $\iota_{(R,\sigma)}$ whose oriented boundaries contain arcs of $\partial(U)$ is $||Int_{R, U}\circ Ext_{R, U}||.$  Let $e=Ext_{R, U}(d),$ and let $f=Int_{R,U}(e)$. Assume $Int_{R, U}(e)=\gamma [\gamma P]^{x(e)}(e^{l_{\delta}}).$ Then
$$Int_{R, U}\circ Ext_{R, U}(d)=\gamma [\gamma P]^{x(e)}\gamma(P\gamma)^{y(d)}(d^{r_{\delta}})=[P\gamma]^{x(e)+y(d)}(d^{r_{\delta}}).$$
Since $d$ and $[P\gamma]^{x(e)+y(d)}(d^{r_{\delta}})$ belong to $\partial(U).$ We have the following claim.

\textbf{Claim:} $x(e)+y(d)$ is the least such integer that $[P\gamma]^{x(e)+y(d)}(d^{r_{\delta}})= f^l$ for $f\in \partial(U).$

Suppose that $[P\gamma]^m(d^{r_{\delta}})= g^{l}$ for some $g\in \partial(U).$ We have $\gamma[P\gamma]^m(d^{r_{\delta}})= g^{r},$ namely,
$$[\gamma P]^{m-y(d)}\gamma[P\gamma]^{y(d)}(d^{r_{\delta}})= g^{r},$$ for some $g\in \partial(U).$

Since $e=Ext_{R, U}(d)=\gamma[P\gamma]^{y(d)}(d^{r_{\delta}})$, hence we have $[\gamma P]^{m-y(d)}(e)= g^{r},$ for some $g\in \partial(U).$ This implies that $m-y(d)\geq x(e),$ i.e., $m\geq x(e)+y(d).$

By the claim, the orbits of
$Ext_{R, U}\circ Int_{R, U}$ are obtained from those of the faces $P\gamma$ of the embedding $\iota_{(R,\sigma)}$ by
deleting all the edges not in $\partial(U).$ The theorem follows.
\end{proof}

\begin{rem}Another way to define $Ext_{R, U}$ and $Int_{R, U}$ are as follows: $$Ext_{R, U}(e)=\gamma [\gamma P]^{x(e)}(e^{l_{\delta}}),\ \ \ \ \  Int_{R, U}(e)=\gamma[P\gamma]^{x(e)} (e^{r_{\delta}})$$
where $x(e)$ and $y(e)$ are the least positive integers such that $[\gamma P]^{x(e)}(e^{l_{\delta}}) \in \{e_{s}^{r_{-}},e_{s}^{r_+}\}$ for some $e_s\in \partial(U)$, and $[P\gamma]^{x(e)}(e^{r_{\delta}})\in \{e_{t}^{l_{+}},e_{t}^{l_{-}}\}$ for some $e_t\in \partial(U)$ respectively.
\end{rem}

Let's illustrate Theorem \ref{Yc18} with two examples.

\begin{example}Let $G=(V,E)$ be the graph of Figure \ref{fig:B3} with the vertex set $V(G)=\{u,v_1,v_2,v_3,v_4,v_5,v_6\}$. \begin{figure}[h]
\centering
\includegraphics[width=4cm,height=3cm]{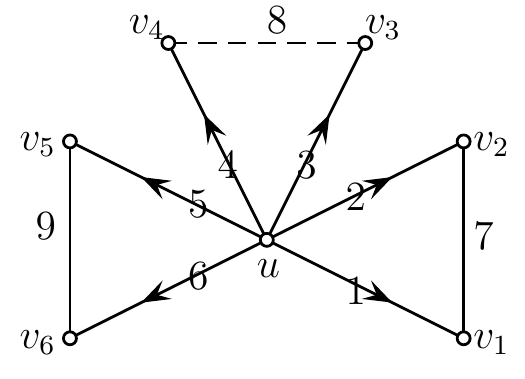}
\caption{}\label{fig:B3}
\end{figure}
Suppose that the rotation system of $G$ at every vertex  is counter-clockwise, i.e., the vertex rotation of $G$ at $u$ is $(123456).$ Let $U=\{u\}$ and $\partial(U)=\{1,2,3,4,5,6\}.$  All the edges of $G$ are untwisted except the edge $v_3v_4,$ i.e., $\sigma(v_3v_4)=-1$ and $\sigma(e)=1,$ for $e\in$ $E(G)\setminus v_3v_4.$

We let

$$\alpha=\prod_{i=1}^{9}(i^{l_+}i^{r_+})(i^{l_-}i^{r_-}),\ \ \beta=\prod_{j=1}^{9}(j^{l_+}j^{r_+})(j^{l_-}j^{r_-}),$$
and
\begin{align*}
P=&(1^{r_+}2^{r_+}3^{r_+}4^{r_+}5^{r_+}6^{r_+})(6^{l_+}5^{l_+}4^{l_+}3^{l_+}2^{l_+}1^{l_+})(7^{r_+}1^{l_-})
(7^{l_+}1^{r_-})(7^{l_-}2^{l_-})(7^{r_-}2^{r_-})(8^{r_+}3^{l_-})\\&(8^{l_+}3^{r_-})(8^{l_-}4^{r_-})(8^{r_-}4^{l_-})
(9^{r_+}5^{l_-})(9^{l_+}5^{r_-})(9^{l_-}6^{l_-})(9^{r_-}6^{r_-}).
\end{align*}
Also, we compute the number of faces of the embedding, we get
\begin{align*}
P\alpha\beta=&(1^{r_+}7^{r_+}2^{l_-}3^{r_+}8^{r_+}4^{r_-}3^{l_+}8^{l_+}4^{l_-}5^{r_+}9^{r_+}6^{l_-})
(1^{r_-}6^{l_+}9^{r_-}5^{r_-}4^{l_+}8^{l_-}3^{l_-}4^{r_+}8^{r_-}3^{r_-}2^{l_+}7^{r_-})\\
&(6^{r_+}9^{l_-}5^{l_-})(6^{r_-}5^{l_+}9^{l_+})(2^{r_+}7^{l_-}1^{l_-})(2^{r_-}1^{l_+}7^{l_+})\\
=&(1^{r_+}\cdots2^{l_-}3^{r_+}\cdots4^{r_-}3^{l_+}\cdots4^{l_-}5^{r_+}\cdots6^{l_-})(6^{l_+}\cdots5^{r_-}4^{l_+}\cdots3^{l_-}4^{r_+}\cdots3^{r_-}2^{l_+}\cdots1^{r_-})\\
&(6^{r_+}\cdots5^{l_-})(5^{l_+}\cdots6^{r_-})(2^{r_+}\cdots1^{l_-})(1^{l_+}\cdots2^{r_-}).
\end{align*}
For exterior permutation, if we start from  the edge $1^{r_{+}}$ in $$(1^{r_+}\cdots2^{l_-}3^{r_+}\cdots4^{r_-}3^{l_+}\cdots4^{l_-}5^{r_+}\cdots6^{l_-}),$$ we get
 $$1\rightarrow 2,\ 3\rightarrow3,\ 4\rightarrow 4,\ 5\rightarrow 6$$ Then by $(6^{r_+}\cdots5^{l_-})$ and $(2^{r_+}\cdots1^{l_-})$, we get
 $$ 6\rightarrow 5,\ 2\rightarrow1,$$
 Thus $$Ext_{R, U}=(12)(3)(4)(56),$$ and the interior permutation is $$Int_{R, U}=(654321).$$ Thus, $$Ext_{R, U}\circ Int_{R, U}=(12)(3)(4)(56)(654321)=(2)(6)(4315).$$


\begin{rem}
The orbit $(2)$ of $(2)(6)(4315)$ represents the face $$(2^{r_+}\cdots1^{l_-})(1^{l_+}\cdots2^{r_-}).$$ Similarly, the orbits $(6)$ and $(4315)$ represent faces $$(6^{r_+}\cdots5^{l_-})(5^{l_+}\cdots6^{r_-})$$ and $$(1^{r_+}\cdots2^{l_-}3^{r_+}\cdots4^{r_-}3^{l_+}\cdots4^{l_-}5^{r_+}\cdots6^{l_-})(6^{l_+}\cdots5^{r_-}4^{l_+}\cdots3^{l_-}4^{r_+}\cdots3^{r_-}2^{l_+}\cdots1^{r_-})$$respectively.
\end{rem}

If we swap the edges $2$ and  $3$ in the permutation $R$ and other bits of $R$ inherit completely, we denote the resulting permutation as $R^{'}$, namely, we let \begin{align*}P^{'}=&(1^{r_+}3^{r_+}2^{r_+}4^{r_+}5^{r_+}6^{r_+})(6^{l_+}5^{l_+}4^{l_+}2^{l_+}3^{l_+}1^{l_+})(7^{r_+}1^{l_-})(7^{l_+}1^{r_-})(7^{l_-}2^{l_-})(7^{r_-}2^{r_-})\\
&(8^{r_+}3^{l_-})(8^{l_+}3^{r_-})(8^{l_-}4^{r_-})(8^{r_-}4^{l_-})
(9^{r_+}5^{l_-})(9^{l_+}5^{r_-})(9^{l_-}6^{l_-})(9^{r_-}6^{r_-}).
\end{align*}
If we start from  the edge $3^{r}$, then $$Ext_{R^{'},U}=(321465)$$ and the interior permutation is $$Int_{R^{'},U}=(654231).$$ We have $$Ext_{R^{'},U}\circ Int_{R^{'},U}=(63415)(2).$$

\end{example}

\begin{example}Let $G=(V,E)$ be the graph of Figure \ref{fig:H6} with the vertex set $V(G)=\{u,v,w,x,y,z\}$. \begin{figure}[h]
\centering
\includegraphics[width=5cm,height=3.5cm]{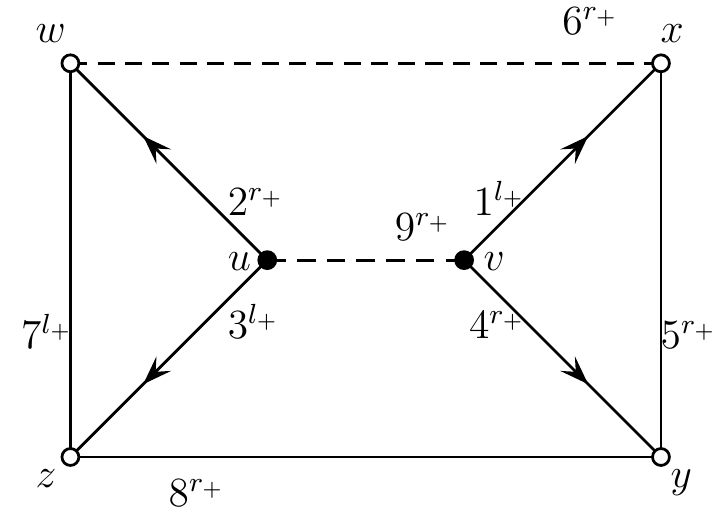}
\caption{}\label{fig:H6}
\end{figure}
Suppose that the rotation system of $G$ at every vertex  is counter-clockwise. 
Let $U=\{u,v\}$ and $\partial(U)=\{1,2,3,4\}.$  All the edges of $G$ are untwisted except the edges $uv,$ and $wx,$ i.e., $\sigma(uv)=\sigma(wx)=-1$ and $\sigma(e)=1,$ for $e\in$ $E(G)\setminus \{uv,wx\}.$ Let

$$\alpha=\prod_{i=1}^{9}(i^{l_+}i^{r_+})(i^{l_-}i^{r_-}),\ \ \beta=\prod_{j=1}^{9}(j^{l_+}j^{r_+})(j^{l_-}j^{r_-}),$$


and
\begin{align*}
P=&(2^{r_+}3^{r_+}9^{l_-})(3^{l_+}2^{l_+}9^{r_-})(1^{r_+}9^{r_+}4^{r_+})(1^{l_+}4^{l_+}9^{l_+})(6^{r_-}2^{r_-}7^{r_-})(6^{l_-}7^{l_-}2^{l_-}) \\ &(6^{r_+}1^{l_-}5^{l_-})(6^{l_+}5^{r_-}1^{r_-})(7^{r_+}8^{r_+}3^{r_+})(7^{l_+}3^{r_-}8^{l_+})(5^{r_+}4^{l_-}8^{l_-})(5^{l_+}8^{r_-}4^{r_-}).
\end{align*}

An exterior permutation for the embedding of $G$ is  $$Ext_{R, U}=(1432),$$ and an interior permutation is as the following $$Int_{R, U}=(14)(32).$$ Thus, $$Ext_{R, U}\circ Int_{R, U}=(13)(2)(4),$$ namely, the number of faces containing the edges of $\partial(U)$ is $3.$

\end{example}

It's easy to see that if $U=\{u\}$, then $Int_{R, u}$ is the counter-clockwise rotation of the edges which incident with $u.$ Furthermore, the following property follows..

\begin{lemma}\label{indep}Let $U=\{u_1,u_2,\ldots,u_k\}$ be an independent set of vertices of a connected signed graph $\Sigma,$ then, for any embedding of $\Sigma,$ $$Int_{R, U}=R^1R^2\cdots R^k,$$ where $R^i$ is the rotation of the edges which incident with $u_i$ for $i=1,2,\ldots,k.$
\end{lemma}

\section{ the average Euler-genus of the vertex-amalgamation
of signed graphs}

Let $H$ be a subgraph of $G$, the \textit{induced rotation system} $R$ of $G$ on $H$ is obtained by deleting all edges of $G-H$ from the rotation system $R.$

For $k\geq 2$, let $\Sigma^i$ ($1\leq i\leq k$) be a set of disjoint signed graphs, each $\Sigma^i$ contains a common set
$U=\{u_1,u_2,\ldots,u_l\}$ of vertices. Then we denote the \textit{vertex-amalgamation} of the signed graphs  $\Sigma^1,\Sigma^2,\ldots,\Sigma^k$ over $U$ by $\Sigma=\bigcup_{U}\Sigma^i$ as the signed graph obtained by identifying  distinct copies of $U.$ For minimum genus and maximum genus of vertex-amalgamation of graphs, see \cite{Alp73,Arc86a,Arc86b,CGMT20,DGH81,Sta82}, etc.

For $d\geq 0,$ the \textit{harmonic sum} $h_d$ is given by  $$h_0=0,h_1=1$$ and $$h_d=1+1+\frac{1}{2}+\cdots+\frac{1}{d-1}$$ for $d\geq 2.$

 The following theorem was proved in \cite{Sta91a}.
\begin{thm}\cite{Sta91a}\label{stahl}
Let $(P, \Pi)$ be a permutation-partition pair such that $\Pi =
(\Pi_1, \Pi_2, \ldots, \Pi_k)$ and $d_i=|\Pi_i|$ for $i=1,2,\ldots,k$.  If $\mu$ denotes the mean
of the number of regions in the random embedding of $(P, \Pi)$, then $\mu(P, \Pi)\leq\dsum_{i=1}^{k}h_{d_i}.$
\end{thm}
The following theorem generalizes Stahl's bound \cite{Sta95} for the average Euler-genus of the vertex-amalgamation of graphs to signed graphs. The ideas are the same, we give a detailed proof here for completeness.
 \begin{thm}Let $\Sigma^i$, $i = 1,2, \ldots, k,$  be connected signed graphs, and let
$U=\{u_1,u_2,\ldots, u_l\}$ be an independent set of vertices of  $\Sigma^i$ and  $\Sigma=\bigcup_{U}\Sigma_i.$  Suppose that
$$d_{\Sigma}(u_j) = d_j$$ and $$d_{\Sigma^i}(u_j) = d_{i,j},\ i = 1, \ldots, k, j = 1,\ldots, l.$$ Then
\begin{align*}\label{bound}
-\dsum_{i=1}^{l}h_{d_{i}}  \leq \gamma_{avg}^E(\Sigma)- \dsum_{i=1}^{k}\gamma_{avg}^E(\Sigma^i)-(k-1)(l-2)\leq \dsum_{i=1}^k\dsum_{j=1}^{l}h_{d_{i,j}}.
\end{align*}
\end{thm}

 \begin{proof}We denote $\rho_{avg}(\Sigma)$ the average number of faces in the random embedding of $\Sigma$ into $S_i$ with Euler-genus $i.$ From the Euler formula, $$|V(\Sigma)|-|E(\Sigma)|+\rho_{avg}(\Sigma)=2-\gamma_{avg}^E(\Sigma),$$ thus we only need to prove the following formula
\begin{align*}
 -\dsum_{i=1}^k\dsum_{j=1}^{l}h_{d_{i,j}}
 \leq \rho_{avg}(\Sigma)- \dsum_{i=1}^{k}\rho_{avg}(\Sigma^i)\leq \dsum_{i=1}^{k}h_{d_{i}}.\end{align*}

 If $R$ is a rotation system of $G,$  then $(R,\sigma)$ is a signed rotation system  of $\Sigma=(G,\sigma).$ For $1\leq i\leq k$, let $R^i$ be the induced rotation system of $G$ on $G^i$, i.e.,$\Sigma^i=(G^i,R^i).$

 By Theorem \ref{Yc18}, \begin{align*}
 \rho(\Sigma, R) &= ||Ext_{(R, U)}\circ Int_{(R, U)}|| + |F|,\\
 \rho(\Sigma^i, R^i) &= ||Ext_{(R^i, U)}\circ Int_{(R^i, U)}|| + |F^i|,\ \ (1\leq i\leq k).
 \end{align*}
 where $F$ ($F^i$) consists of all the faces of $(R,\sigma)$ ( $(R^i,\sigma)$ )that contain no arcs of $\partial(U).$ Here $Ext_{(R^i, U)}$ and $Int_{(R^i, U)}$ are permutations of $\partial(U)\cap G^i.$

 It's clear that
  \begin{align*}|F|=\dsum_{i=1}^k |F^i|.\end{align*}

From the discussion above, we have
 \begin{eqnarray}\label{equ:1}
\rho(\Sigma,R)- \dsum_{i=1}^{k}\rho(\Sigma^i,R^i)&=&||Ext_{(R, U)}\circ Int_{(R, U)}||-\dsum_{i=1}^{k}||Ext_{(R^i, U)}\circ Int_{(R^i, U)}|| \\
   &\leq& ||Ext_{(R, U)}\circ Int_{(R, U)}||.\notag
 \end{eqnarray}
 
Let $R$ take all elements in $\Omega$ of $G$  with uniform probability $\frac{1}{|\Omega|}$. Then,
\begin{eqnarray*}
\frac{1}{|\Omega|}\dsum_{R\in\Omega}(\rho(\Sigma,R)- \dsum_{i=1}^{k}\rho(\Sigma^i,R^i)) &\leq& \frac{1}{|\Omega|}
\dsum_{R\in\Omega} ||Ext_{(R, U)}\circ Int_{(R, U)}||
 \end{eqnarray*}
which is equivalent to
 \begin{align}
 \rho_{avg}(\Sigma)- \dsum_{i=1}^{k}\rho_{avg}(\Sigma^i)\leq ||Ext_{(R, U)}\circ Int_{(R, U)}||.\end{align}

By Lemma \ref{indep}, $$Int_{R, U}=R^1R^2\cdots R^k,$$
Then we obtain a permutation-partition pair $(P,\Pi)$, where
$$P = Ext_{(R, U)}$$ and $$\Pi = \{{\partial(U)\cap G^i|i=1,2,\ldots,k}\}.$$ Note that $|\partial(U)\cap G^i|=d_{j}$, and by Theorem \ref{stahl}, the expected value of $||Ext_{(R, U)}\circ Int_{(R, U)}||$ is less than or equal to $$\dsum_{i=1}^{k}h_{d_{i}}.$$

\noindent Since this bound is independent of $Ext_{(R, U)}$ for any $R$ of $G$, thus, \begin{align*} \rho_{avg}(\Sigma)- \dsum_{i=1}^{k}\rho_{avg}(\Sigma^i)\leq \dsum_{i=1}^{k}h_{d_{i}}.\end{align*}

We next  prove the other bound,  again by Equation (\ref{equ:1}), it follows that
\begin{align*}
 \rho_{avg}(\Sigma)- \dsum_{i=1}^{k}\rho_{avg}(\Sigma^i)\geq -\dsum_{i=1}^{k}||Ext_{(R^i, U)}\circ Int_{(R^i, U)}||.\end{align*}
Similarly, we have $$||Ext_{(R^i, U)}\circ Int_{(R^i, U)}||\leq \dsum_{j=1}^{l}h_{d_{i,j}}.$$
Thus
\begin{align}
 \rho_{avg}(\Sigma)- \dsum_{i=1}^{k}\rho_{avg}(\Sigma^i)\geq -\dsum_{i=1}^{k}\dsum_{j=1}^{l}h_{d_{i,j}}.\end{align}
The result follows.
\end{proof}


\end{document}